\documentclass{article}%
\usepackage{amsfonts}
\usepackage{amsmath}
\usepackage{amssymb}
\usepackage{graphicx}%
\providecommand{\U}[1]{\protect\rule{.1in}{.1in}}
\providecommand{\U}[1]{\protect\rule{.1in}{.1in}}
\newtheorem{theorem}{Theorem}

\newtheorem{corollary}[theorem]{Corollary}

\newtheorem{lemma}[theorem]{Lemma}

\newtheorem{proposition}[theorem]{Proposition}
\newtheorem{remark}[theorem]{Remark}

\newenvironment{proof}{\textbf{Proof:}}{\hfill$\blacksquare$}

\newcommand{\bean}{\begin{eqnarray*}}
\newcommand{\eean}{\end{eqnarray*}}
\newcommand{\benu}{\begin{enumerate}}
\newcommand{\eenu}{\end{enumerate}}
\newcommand{\eea}{\end{eqnarray}}
\newcommand{\bea}{\begin{eqnarray}}

\begin{document}

\author{Manuel Guti\'errez\thanks{Departamento de \'Algebra, Geomtr\'ia y Topolog\'ia,
Universidad de M\'alaga, 29071-M\'alaga (Spain), Email:
\texttt{mgl@agt.cie.uma.es}}, Olaf M\"uller\thanks{Fakult\"at f\"ur
Mathematik, Universit\"at Regensburg, D-93040 Regensburg (Germany), Email:
\texttt{olaf.mueller@ur.de}}}
\title{Compact Lorentzian holonomy \thanks{The first author was supported in part by
MEYC-FEDER Grant MTM2013-41768-P and Junta de Andaluc\'ia research group
FQM-324.}}
\maketitle

\begin{abstract}
We consider (compact or noncompact) Lorentzian manifolds whose holonomy group
has compact closure. This property is equivalent to admitting a parallel
timelike vector field. We give some applications and derive some properties of
the space of all such metrics on a given manifold.

\end{abstract}



\section{Introduction}

It is well known that a structure group reduction of the frame bundle encodes
the existence of a geometric structure on the manifold. If, moreover, it
contains the holonomy group of a given connection $\nabla$, the geometric
structure is $\nabla$-parallel, \cite[Propositions 5.6 and 7.4]%
{KobayNomiz1963vI}. The most familiar example is the existence of a
semi-Riemannian metric which is equivalent to a reduction of the structure
group to $O_{\nu}(n)$. Metricity of the Levi-Civita connection implies that
its holonomy group is contained in $O_{\nu}(n)$. Another classical example is
a $2n$-dimensional K\"{a}hler manifold. It has holonomy group contained in
$U(n) $. In fact, $U(n)=GL(n,\mathbb{C})\cap Sp(n,\mathbb{R})\cap O(2n)$, and
this means that the manifold has a complex structure and a symplectic
structure which are parallel and adapted to a Riemannian metric.

In (oriented) Riemannian geometry, the generic holonomy is the (special)
orthogonal group, so noncompact (i.e., non-closed) holonomy implies the
presence of a parallel geometric structure. Simply connected Riemannian
manifolds have compact holonomy group because it coincides with its restricted
holonomy group, which is well known to be compact, \cite{BorLic52}. On the
other hand, the question of the existence of a compact Riemannian manifold
with noncompact holonomy was solved in \cite{Wil99} where the author showed
the existence of such manifolds and studied their structures. In fact, a
compact Riemannian manifold with noncompact holonomy has a finite cover that
is the total space of a torus bundle over a compact manifold, and its
dimension is greater or equal than 5.

Lorentzian holonomy groups have attracted much attention in the last years ---
for overview articles on this topic we refer to \cite{hB2012}, \cite{aGtL}.
\cite{aG2015}. Now, the situation in Lorentzian manifolds is similar but
slightly different because the generic holonomy is the Lorentz group which is
noncompact. It is natural to ask the analogous question: can we describe the
Lorentzian manifolds which have compact holonomy?

Noncompactness of the holonomy group is responsible for noncompleteness in
some compact Lorentzian manifods, as in the Clifton-Pohl torus. The
relationship between holonomy and completeness is in general not well known,
see e.g. \cite{LeistnerSchliebner13} were the authors study the case of
compact pp-waves. It is related to undesirable identifications of singular
points in b-singularity theory, \cite{Schm71}. In fact, in \cite{AmoGut99} it
was shown that in the four dimensional Friedmann model of the Universe, which
have noncompact holonomy, big bang and big crunch are the same point in the
b-boundary. On the other hand, compactness of the holonomy group has been used
in \cite{GHP91} to define so-called Cauchy singular boundaries in space-times.
Later, one of the authors (M. G.), using the fundamental observation that both
a Lorentzian metric $g$ and its flip Riemannian metric around a parallel
timelike vector field induce the same Levi-Civita connection, proved that
compactness of the holonomy group implies that the Cauchy singular boundary of
the manifold is homeomorphic to its b-boundary, \cite{Gut09}.

In this article we identify Lorentzian manifolds (compact or not)\ whose
holonomy groups have compact closure (Theorem \ref{Teorema1}), and draw some
consequences. For example, we characterize the case $Hol\subset SO(m-k)$,
$k\in\{1,...,m\}$ (Theorem \ref{ParallelCouple}) and use it to show that in
the category of complete semi-Riemannian manifolds, if they can be decomposed
in a direct product in a weak sense, then it is generically unique between
direct product decomposition (weak or not), and this property fails only in
2-dimensional Minkowski or Euclidean spaces (Theorem \ref{Teorema2}). This is
a rigidity type result that cannot be achieved directly from the uniqueness of
the De Rham-Wu theorem.

Motivated by conceptual questions around the Lorentzian Einstein equation and
led by the characterization given in Theorem \ref{Teorema1}, we then initiate
a study on various topologies on the set $G(M)$ of globally hyperbolic
Lorentzian metrics on a manifold $M$, in topologies induced by usual
topologies on the space $\mathrm{Bil}(M)$ of bilinear forms on $M$. We find
that the closure of $K(M)$, the subset of metrics with precompact holonomy, in
the compact-open topology, consists of metrics with parallel causal vector
fields (see Theorem \ref{closure}). Moreover, as each connected component of
$G(M)$ with the $C^{0}$-fine topology is formed by metrics with diffeomorphic
Cauchy surfaces, we can show that $G(\mathbb{R}^{n})$ with $n\geq4$ has
uncountably many componentes each of which intersects $K(M)$. Finally, if
$g\in K(M)$ is timelike complete, then any other metric in the same $C^{1}%
$-fine path connected component of $K(M)$ is isometric to $g$, Corollaries
\ref{Cauchy} and \ref{Corollary1}.

We wish to thank Thomas Leistner and two anonymous referees for helpful comments on this paper.

\section{Compact holonomy}

We assume all manifolds to be connected. Let $(M,g)$ be a semi-Riemannian
manifold with signature $\nu$. We denote $Hol (M,g)$ and $Hol^{0} (M,g)$ its
holonomy group and its restricted holonomy group, respectively. We drop $M$ if
no confusion is possible. $Hol_{p} (M,g)$ and $Hol^{0}_{p}(M,g)$ will refer to
the resp. holonomy group at a given point (whereas the corresponding objects
without specifying a point are, strictly speaking, only \emph{equivalence
classes of subgroups of $Gl(n)$ or of representations}). Observe that whereas
in the Riemannian case, the reduced holonomy is always a closed subgroup of
$SO(n)$ \cite{BorLic52}, there are examples of simply connected Lorentzian
manifolds with non-closed holonomy group, \cite{BerIke1993}.

The following theorem identifies Lorentzian manifolds with holonomy contained
in a compact group. We use the following well known lemma

\begin{lemma}
\label{Lema1}The map
\[%
\begin{array}
[c]{ccc}%
\pi_{1}(M,p) & \overset{j}{\longrightarrow} & Hol (M,g)/Hol^{0} (M,g)
\end{array}
\]
given by $j([\gamma])=[p_{\gamma}]$ (where $p_{\gamma}$ is the parallel
transport along $\gamma$) is a surjective group morphism.
\end{lemma}

The following theorem could be derived as a consequence of the well know
fundamental principle (see \cite[10.19]{Besse1987}), combined with the
classification of subgroups of the Lorentz group, and the above lemma.
However, for later use, we prefer to give a proof using the Haar measure.

\begin{theorem}
\label{Teorema1} Let $(M,g)$ be a time oriented Lorentzian manifold.

\begin{enumerate}
\item The holonomy group is relatively compact (and thereby contained in a
compact group) if and only if it admits a timelike parallel vector field.

\item If $(M,g)$ admits a timelike parallel vector field and $\pi_{1}(M,p)$ is
finite, then its holonomy group is compact.
\end{enumerate}
\end{theorem}

\begin{remark}
\label{R 1}Of course, a parallel timelike vector field induces a local product
splitting of the manifold. But a local or global splitting does not suffice to
imply that the holonomy is contained in a compact group, see Remark
(\ref{R 4}) below for a counterexample.

It is also not true that the holonomy is always compact if there is a parallel
timelike vector field, see Remark (\ref{R 2}).
\end{remark}

\begin{remark}
Note that the proof below does not work in higher signature because the
stabilizer of a nontrivial vector is not compact, so it is not possible to
distinguish a non trivial parallel vector field.
\end{remark}

\begin{remark}
\label{R 2} The hypothesis on the finiteness of $\pi_{1}(M,p)$ is necessary to
ensure compact holonomy. A counterexample are the direct products
$(\mathbb{S}^{1},-dt^{2})\times(T,g_{0})$ and $(\mathbb{R}^{1},-dt^{2}%
)\times(T,g_{0})$ whose holonomy groups equal the one of $(T,g_{0})$ and we
can choose it with noncompact holonomy, as soon as $\vert\pi_{1} (T) \vert=
\infty$.
\end{remark}

\begin{proof}
As $Hol(M,g)$ is contained in a compact subset, its closure $C$, which is a
subgroup as well, is compact, and thus carries a bi-invariant Haar measure
$\mu$. Now let a point $p\in M$ be given. We want to construct a timelike
vector $v\in T_{p}M$ invariant under $Hol_{p}$. To that purpose, choose a
future timelike vector $v_{0}\in T_{p}M$ at will and define $v:=\int_{Hol_{p}%
}h(v_{0})d\mu(h)$. The integral exists as the Haar measure of the compact
group $C$ is finite and the action is continuous. Now, given any $k\in
Hol_{p}$, we compute%
\[
kv=\int_{Hol_{p}}k\circ h(v_{0})d\mu(h)=\int_{Hol_{p}}h(v_{0})k^{\ast}%
d\mu(h)=\int_{Hol_{p}}h(v_{0})d\mu(h)=v,
\]
so indeed $v$ is invariant, and it is timelike, as the integrand consists in
timelike future vectors and those form a convex set. And now, using the
parallel transport $P_{c}$ along a curve $c$, we have that $P_{c}(v)=P_{k}(v)$
if $c(0)=p=k(0)$ and $c(1)=k(1)$, because $P_{k}$ is an isomorphism and
$P_{k}^{-1}\circ P_{c}(v)=P_{ck^{-1}}(v)=v$ as $P_{ck^{-1}}\in Hol_{p}$. Thus
there is a well-defined way to extend $v$ to a parallel future timelike vector
field $V$. The second item follows easily from the fact that $Hol_{0}$ is
connected \cite{BorLic52}, from Lemma \ref{Lema1} and the fact that there is
no finite quotient of a noncompact group by a compact one.
\end{proof}

\vspace{0.5cm}

Let us show a statement adapted to a submanifold. We define,
for a submanifold $S\subset M$, and a point $x\in S$%
\[
Hol(S,M,g):=\{P_{c}\ /\ c(0)=x=c(1),\ c^{\prime}(s)\in TS\ \forall s\in
\lbrack0,1]\}
\]
where $P_{c}$ is the parallel transport w.r.t. the connection of the ambient
manifold $M$, and call $Hol(S,M,g)$ the \emph{relative holonomy of }$S$. The
independence on the point $x$ is up to conjugation, just as in the case of
usual (absolute) holonomy. 

\begin{theorem}
\label{RelativeHolonomy} If $S$ is a spacelike totally geodesic submanifold of a
time-oriented Lorentzian manifold $(M,g)$ and if $Hol (S,M,g)$ (resp. $Hol^{0}
(S,M,g)$) is precompact, then the normal bundle of $S$ contains a section that
is invariant under $Hol (S,M,g)$ (resp. under $Hol^{0} (S,M,g)$).
\end{theorem}

\begin{proof}
The hypothesis on $Hol (S,M,g)$ fixes a temporal vector field $V$ like in the
proof above, which we can assume to be future. As $TS$ is invariant by $Hol (S,M,g)$, 
we know that $W:=pr_{TS}%
^{g}V$ and $\tilde{V} := V - W $ are fixed by $Hol (S,M,g)$ as well. As the
latter is normal, this proves the claim.
\end{proof}

\begin{theorem}
\label{ParallelCouple} Let $(M,g)$ be an oriented and time oriented
$m$-dimensional Lorentzian manifold. Then, $Hol\subset SO(m-k)$,
$k\in\{1,...,m\}$ if and only if $M$ admits an orthonormal system
$\{V_{1},V_{2},...,V_{k}\}$ formed by parallel vector fields, with $V_{1}$ timelike.
\end{theorem}

\begin{proof}
It is clear that it is true for $k=1$ after Theorem \ref{Teorema1} and as the
conjugacy class of $SO(m-1) $ is the maximal compact conjugacy class of
subgroups of $SO(1,m-1)$. Suppose it is true for $k-1$. If $\{e_{1}%
,e_{2},...,e_{k-1}\}$ is an ortonormal system in $T_{p}M$ with $e_{i}%
=V_{i}(p)$, then, $Hol$ keeps $V_{1},...,V_{k-1}$ invariant and $SO(m-k)$ is
the stabilizer of a timelike $k $-dimensional subspace $L^{k}$ of $T_{p}M$
containing the above system$\{e_{1},e_{2},...,e_{k-1}\}$. We can complete this
system to an orthonormal basis of $L^{k}$ choosing a spacelike vector
$e_{k}\in L^{k}$. This induce an $Hol$ invariant vector field $V_{k}$ showing
that the theorem is true for $k$.

Conversely, let $U\subset T_{p}M$ be the subspace generated by $\{V_{1}%
,V_{2},...,V_{k}\}_{p}$. Each element $h\in Hol$ decomposes as%
\[
id\oplus h_{2}:U\oplus U^{\bot}\rightarrow U\oplus U^{\bot}%
\]
where $h_{2}$ acts as an isometry on $U^{\bot}$.
\end{proof}

\begin{remark}
\label{R 4}The existence of a Lorentzian manifold with a timelike parallel
vector field $V$ and $Hol$ noncompact is clear in the noncompact case because
we turn the question into a well-known Riemannian one using the flip metric%

\begin{equation}
g_{R}(X,Y)=g(X,Y)+2g(X,V)g(Y,V) \label{MetricaRiemann}%
\end{equation}
with $g(V,V)=-1$ (and thus $g_{R}(V,V) =1$) which share the Levi-Civita
connection with the Lorentzian metric $g $. The compact case is more involved,
but it can be solved using the results in \cite{Wil99}. We consider three cases

\begin{itemize}
\item $\dim M\geq6$, the above example $M=\mathbb{S}^{1}\times T$ with $T$
compact and $Hol (T)$ noncompact shows that they exist, but in this case we
know that $\dim T\geq5$.

\item $\dim M\leq4$, the presence of a parallel timelike vector field allows
us construct the Riemannian flip metric on $M$, and this implies that the
holonomy is compact.

\item $\dim M=5$. We can not apply neither of the above direct arguments, but
the Wilking example provides one.

Consider a semidirect product $S=\mathbb{R}^{4}\rtimes\mathbb{R}$, and a
discrete cocompact subgroup $\Lambda$ which acts as deck transformation group
of the covering $p:S\rightarrow\Lambda\backslash S$ by left translations. The
group $S$ admits a left invariant metric $g=\left\langle \cdot,\cdot
\right\rangle \times g_{1}$ where $\left\langle \cdot,\cdot\right\rangle $ is
the euclidean metric in $U=\mathbb{R}^{2}\times\{0\}$ and $g_{1}$ is a left
invariant metric on $S_{1}=\{0\}\times\mathbb{R}^{2}\rtimes\mathbb{R}$. The
Wilking example is the quotient $\Lambda\backslash S$ with the induced metric
with the choice of two parameters. It has non compact holonomy group. The left
invariant vector field $V\in\mathfrak{X}(S)$ defined by $(1,0,0,0,0)\in
\mathfrak{s}$, where $\mathfrak{s}$ is the Lie algebra of $S$, is invariant by
$\Lambda$ so it defines a vector field $V\in\mathfrak{X}(\Lambda\backslash
S)$. It is clear that both vector fields are parallel.

Using the flip metric in (\ref{MetricaRiemann}) we get a Lorentzian metric on
a compact manifold $\Lambda\backslash S$ with $V$ a timelike parallel vector
field and non compact holonomy.
\end{itemize}
\end{remark}

As an application of Theorem \ref{Teorema1}, we see directly that some kind of
manifolds do not admit Lorentzian metrics with relatively compact holonomy,
for example odd spheres (which \emph{do} admit Lorentzian metrics because of
vanishing Euler number but which are not direct products).

We compare the holonomy groups in a covering space. Let $\pi:M\longrightarrow
B$ be a semi-Riemannian covering, so both $M$ and $B$ have the same restricted
holonomy group.

\begin{lemma}
\label{Lema3}Let $\pi:M\longrightarrow B$ be a semi-Riemannian covering map.

\begin{enumerate}
\item The map $\pi^{\#}:Hol (M)\longrightarrow Hol (B)$ given by $\pi
^{\#}(P_{\gamma})=P_{\pi\circ\gamma}$, is a Lie group monomorphism.

\item If $Hol (B)$ is compact, then $Hol (M)$ is also compact.

\item If $Hol (M)$ is compact and $\pi_{1}(B,p)$ is finite, then $Hol (B)$ is
also compact.
\end{enumerate}
\end{lemma}

\begin{proof}
Observe that $P_{\gamma}P_{\beta}=P_{\beta\gamma}$ and $P_{\pi(\beta\gamma
)}=P_{\pi(\gamma)}P_{\pi(\beta)}$ for any couple of lassos $\gamma,\beta$ at
$p$. On the other hand, $\pi$ is a local isometry so $\pi_{\ast p}P_{\gamma
}=P_{\pi(\gamma)}\pi_{\ast p}$. Thus if $P_{\gamma}=P_{\beta}\in Hol^{M}$, we
have $e=P_{\beta^{-1}\gamma}$ being $e$ the identity element, and applying
$\pi_{\ast_{p}}$ implies that $P_{\pi(\gamma)}=P_{\pi(\beta)}$. This shows
that $\pi^{\#}$ is well defined.

\begin{enumerate}
\item[1.] We see that it is a morphism using $P_{\gamma}P_{\beta}%
=P_{\beta\gamma}$ and $P_{\pi(\beta\gamma)}=P_{\pi(\gamma)}P_{\pi(\beta)}$. To
see that it is injective use $\pi_{\ast p}P_{\gamma}=P_{\pi(\gamma)}\pi_{\ast
p}$.

\item[2.] Observe that $Hol^{0} (B)=Hol^{0} (M)$ is compact, so the connected
components of $Hol (M)$ are diffeomorphic to $Hol^{0} (B)$ and $Hol (M)$
itself can be identified to its image by $\pi^{\#}$ in $Hol (B)$. Finally,
$Hol (B)$ has a finite number of connected components because it is compact.

\item[3.] Note that the hypothesis implies $\#Hol (B)/Hol^{0} (B)<\infty$ by
Lemma \ref{Lema1}, and $Hol^{0} (B)=Hol^{0} (M)$ is compact, thus $Hol (B)$ is
also compact.
\end{enumerate}
\end{proof}

Given $u,v\in T_{p}M$ where $u$ is a null vector, it is defined the null
sectional curvature of the degenerate plane $\pi=span\{u,v\}$ as%
\[
\mathcal{K}_{u}(\pi)=\frac{g(R_{uv}v,u)}{g(v,v)}.
\]

It depends on the null vector $u\in T_{p}M$, but once it is fixed, it is a map
on degenerate planes in $T_{p}M$ containing $u$. If we fix a timelike vector
field $U$, we can see $\mathcal{K}_{U}$ as a map on the subset of degenerate
planes in the Grassmannian of planes in $TM$, defining $\mathcal{K}_{U}%
(\pi):=\mathcal{K}_{u}(\pi)$ where $u\in\pi$ is the unique null vector such
that $g(u,U)=1$. There are examples where $\mathcal{K}_{U}$ is in fact a map
from $M$, that is, it does not depend on the choice of degenerate plane
$\pi\subset T_{p}M$ but just on the point $p$ itself. In this case we say that
it is a pointwise function. It is a strong condition, in some sense similar to
the well known Schur Lemma in the Riemannian case, see \cite{Harris1985} for
details. The sign of $\mathcal{K}_{U}$ does not depend on the chosen vector
field $U$. In fact, if $U^{\prime}$ is another timelike vector field, and
$u\in T_{p}M$ is the unique null vector in $\pi$ with $g(u,U)=1$,
\[
\mathcal{K}_{U}(\pi)=g(u,U^{\prime})^{2}\mathcal{K}_{U^{\prime}}(\pi).
\]

So it is reasonable to speak of positive null sectional curvature for all
degenerate planes, \cite{Har82}.

The following result shows that null curvature can help to determine a
Lorentzian holonomy.

\begin{proposition}
Let $(M,g)$ be a complete and non-compact Lorentzian manifold with $m=\dim
M\geq4$ such that the null sectional curvature is a positive pointwise
function. If the holonomy group is contained in a compact group, then
$Hol(M)=SO(m-1)$ or $O(m-1)$.
\end{proposition}

\begin{proof}
A suitable finite covering $\widetilde{M}$ of $M$ is orientable and time
orientable, so Lemma \ref{Lema3} and Theorem \ref{Teorema1} tell us that
$\widetilde{M}$ admits a timelike parallel vector field $U$. In particular, we
can use $U$ as a geodesic reference frame in the sense of \cite{GutOle09} to deduce that $\widetilde{M}$ is a
direct product $\mathbb{R}\times L$ where the second factor is a quotient of
the usual sphere $\mathbb{S}^{m-1}$ of constant positive curvature, see
\cite[Proposition 5.4]{GutOle09}. The fact that $L$ is a quotient of
$\mathbb{S}^{m-1}$ and Lemma \ref{Lema3} again implies $SO(m-1)\subset
Hol(\widetilde{M})\subset Hol(M)$. By hypothesis, $Hol(M)$ is contained in a
compact group, in particular in a maximal compact one, that is, in a copy of
$O(m-1)$.
\end{proof}

Let us consider another consequence of Theorem \ref{Teorema1}. It is a
well-known result by Marsden \cite{Marsden1972} that a compact homogeneous
semi-Riemannian manifold is geodesically complete (whereas the same is not
true omitting the condition of homogeneity). Using a suitable covering and
\cite{RS}, we have

\begin{corollary}
Let $M$ be a compact manifold and let $g$ be a Lorentzian metric on $M$ with
precompact holonomy. Then $(M,g)$ is geodesically complete.
\end{corollary}

This result is a particular case of a more general result valid for holonomy
groups defined for an arbitrary connection, see \cite{Ake2016}.

Inspired in \cite{GutOle12}, we can show that Euclidean and Minkowski plane
are the unique semi-Riemannian manifolds with the property that they admit
another direct product decomposition with non degenerate properties.

A semi-Riemannian manifold is called \emph{(locally) indecomposable} iff its
holonomy group is \emph{weakly irreducible}, i.e. iff the latter does not
admit non-trivial and nondegenerate invariant subspaces by the holonomy group
in any tangent space. In the Riemannian case this notion coincides with the
usual notion of irreducibility. Given a\ manifold $M=M_{1}\times M_{2}$, we
call $M_{i}(p)$ the tangent space at $p$ of the leaf of the $i$-th canonical
foliation through $p\in M$.

\begin{theorem}
\label{Teorema2}Let $M=M_{1}\times M_{2}$ be a complete semi-Riemannian direct
product with $M_{i}$ indecomposable. Suppose that $M$ admits another
decomposition as a direct product $M=L_{1}\times L_{2}$ (with $L_{1}%
\not =M_{i}$), and $M_{i}(p)\cap L_{j}(p)$ zero or non degenerate. Then
$M=\mathbb{R}^{2}$, the euclidean or Minkowski plane.
\end{theorem}

\begin{proof}
Suppose that $\dim M_{1}=k$, and the signature of $M_{i}$ is $\nu_{i}$, such
that the signature of $M$ is $\nu=\nu_{1}+\nu_{2}$. Let $i:O_{\nu_{1}%
}(k)\longrightarrow O_{\nu}(m)$ and $j:O_{\nu_{2}}(m-k)\longrightarrow O_{\nu
}(m)$ be the natural inmersions $i(c)=%
\begin{pmatrix}
c & 0\\
0 & I_{m-k}%
\end{pmatrix}
$, $j(d)=%
\begin{pmatrix}
I_{k} & 0\\
0 & d
\end{pmatrix}
$. We call $G=i(O_{\nu_{1}}(k))j(O_{\nu_{2}}(m-k))$. It is clear that if $M$
is a direct product $M_{1}\times M_{2}$, its holonomy group $H$ is reducible
to a subgroup of $G$, that is, $H=H_{1}H_{2}$ with $H_{1}\subset i(O_{\nu_{1}%
}(k))$ and $H_{2}\subset j(O_{\nu_{2}}(m-k))$. Let $\pi:OM\longrightarrow M$
be the orthonormal frame bundle. Call%
\begin{align*}
E  &  =\{r\in OM\ /\ r:\mathbb{R}^{m}\longrightarrow T_{\pi r}%
M\ such\ that\ carries\ adapted\\
&  \ basis\ of\ \mathbb{R}^{k}\times\mathbb{R}^{m-k}%
\ to\ adapted\ basis\ of\ M_{1}\times M_{2}\}.
\end{align*}

With respect to the decomposition $M=L_{1}\times L_{2}$, fixed an element
$r\in E$, there exists another decomposition of $\mathbb{R}^{m}$ as a direct
product $S_{1}\times S_{2}$ such that $r$ carries an adapted basis of
$S_{1}\times S_{2}$ to an adapted basis of $L_{1}\times L_{2}$.

Both tuples of foliations in $M$ are invariant by parallel transport, that is,
the subspaces $\mathbb{R}^{k}$, $\mathbb{R}^{m-k}$, $S_{1}$ and $S_{2}$ of
$\mathbb{R}^{m}$ are invariant by the holonomy group $H$.

Given $h\in H_{1}$, we can write $h=%
\begin{pmatrix}
c & 0\\
0 & I
\end{pmatrix}
$ with $c\in O_{\nu_{1}}(k)$, and if we call $(x_{1},x_{2})$ the components of
$x=\pi(r)\in M$ in $M_{1}\times M_{2}$ and $(x_{1}^{\prime},x_{2}^{\prime})$
its components in $L_{1}\times L_{2}$, we have the following two ways in which
we can write the composition $r\circ h$%
\[%
\begin{array}
[c]{ccccc}%
\mathbb{R}^{k}\times\mathbb{R}^{m-k} & \overset{h}{\longrightarrow} &
\mathbb{R}^{k}\times\mathbb{R}^{m-k} & \overset{r}{\longrightarrow} &
T_{x_{1}}M_{1}\times T_{x_{2}}M_{2}\\
S_{1}\times S_{2} & \overset{h}{\longrightarrow} & S_{1}\times S_{2} &
\overset{r}{\longrightarrow} & T_{x_{1}^{\prime}}L_{1}\times T_{x_{2}^{\prime
}}L_{2}.
\end{array}
\]

Given $(u,0)\in S_{1}\times\{0\}$, we have $h(u,0)\in S_{1}\times\{0\}$
because $S_{1}$ is invariant by $H$. On the other hand if we write $u$ with
its components in the other decomposition, $u=(u_{1},u_{2})\in\mathbb{R}%
^{k}\times\mathbb{R}^{m-k}$, we have $h(u)=(cu_{1},u_{2})$ and%
\[
u-h(u)=(u_{1}-cu_{1},0)\in(S_{1}\cap\mathbb{R}^{k})\times\mathbb{R}^{m-k}.
\]

By hypothesis $S_{1}\cap\mathbb{R}^{k}$ is zero or a non degenerated subspace
of $\mathbb{R}^{k}$ invariant by $H$, but the holonomy of $M_{1}$ is weakly
irreducible, so it must be zero, thus $H_{1}=\{1\}$. A similar argument for
$H_{2}$ implies that $H=\{1\}$. Theorem \ref{ParallelCouple} implies that $M$
admits a global orthonormal basis formed by parallel vector fields
$E_{1},...E_{m}$. By completeness, the universal covering $\widetilde{M}$
splits as $\mathbb{R}^{m}$ with a flat metric. The group of deck
transformation preserves the parallel basis, otherwise $H$ would not be
trivial, thus $M_{i}$ is a product of $m_{i}=\dim M_{i}$ factors each one
being $\mathbb{R}$ or $\mathbb{S}^{1}$, but the holonomy of $M_{i}$ being
weakly irreducible we have $m_{i}=1$, therefore $m=\dim M=2$. The only
complete flat surfaces that admits two different structures as a direct
product are the euclidean and the Minkowski plane.
\end{proof}

Note that in this proof, we do not suppose a priori that $L_{i}$ must also be
indecomposable nor $M$ simply connected. This is of crucial importance because
if we suppose them, the uniqueness of the decomposition in the Theorem of de
Rham-Wu can be used to give a direct proof, \cite{Wu64}, \cite[Apendix
I]{Wu67}, \cite{Chen2014}. In fact, $M$ should be isometric to $\mathbb{R}%
^{m}$ with flat metric and $H=\{1\}$. So $M_{i}$ is a direct product of
$m_{i}=\dim M_{i}$ factors $\mathbb{R}^{1}$, but $M_{i}$ indecomposable
imposes $m_{i}=1$. Thus $\dim M=2$.

\section{Topologies on the space of all metrics with precompact holonomy}

Having made several statements about single metrics with precompact holonomy,
let us try to explore the topology of the space of \emph{all} time-oriented
metrics on an orientable manifold $M$ that have precompact holonomy, in
analogy to the situation in positive curvature. However, it will turn out in
the following that much of its topology is hidden behind the not quite
accessible topology of the space of Lorentzian metrics. In the light of
applications like the Einstein equation considered as a variational problem,
it is outermost desirable to construct an appropriate topology on the space of
Lorentzian metrics. Several topologies on that space and on related spaces
have been considered. In a row of articles, Bombelli, Meyer, Noldus and
Sorkin, e.g., introduced topologies on the quotient $\mathrm{Lor}%
(M)/\mathrm{Diff}(M)$ based on a splitting between the conformal and the
volume part (for an overview, see \cite{jN}), but unfortunately, this topology
is not a manifold topology in general. As we are ultimately interested in
variational problems, and thus look for a manifold topology on the space
$\mathrm{Lor}_{+}(M)$ of time-oriented Lorentzian metrics, the simplest choice
is the subspace topology with respect to a topological vector space topology
on the space $\mathrm{Bil}(M)$ of bilinear forms on $M$. Let%
\[
PT (M):=\{g\in\mathrm{Lor}_{+}(M)\ /\ \overline{Hol_{g}}\ \mathrm{compact}\}
\]
thus, following Theorem \ref{Teorema1}, $PT(M)$ is the set of time-oriented
Lorentzian metrics with a parallel timelike vector field.

We define $G(M)$ to be the space of all globally hyperbolic metrics and $C(M)
$ to be the set of all \emph{causally complete}, i.e. timelike and lightlike
complete, metrics on $M$.

First of all we want to compare the different possible topologies on $PT(M)$
(understood as a subset of $\mathrm{Lor}(M)$). On one hand, if $M$ is
noncompact, only a topology at least as fine as the $C^{0}$-fine (Whitney)
topology on $\mathrm{Bil} (M)$ ensures that $\mathrm{Lor}(M)$ is an open
subset of $\mathrm{Bil}(M)$. On the other hand, as we want to be able to
define parallel vector fields, all metrics should at least be $C^{1}$, and
thus, the fact that we want to have a \emph{complete} vector space topology on
$Bil(M)$, recommends us to choose a topology at least as strong as the $C^{1}%
$-compact-open topology. First of all, for $PC (M)$ being the set of
time-oriented Lorentzian metrics with a parallel causal vector field, we
observe that we can control the closure of $PT(M)$ in terms of $PC(M)$:

\begin{theorem}
\label{closure} [Closure of $PT(M)$]Let $M$ be a manifold.

\begin{enumerate}
\item {In any topology finer or equal to the $C^{1}$-compact-open\textbf{\ }%
topology, the closure of the set $PT(M)$ is contained in $PC(M)$.}

\item {If $M$ is diffeomorphic to }$\mathbb{R}${$\times S$ for some manifold
$S$ (to ensure $G(M)\neq\emptyset$), then for $E:=PC(M)\cap G(M)\cap C(M)$ and
$F:=PT(M)\cap G(M)\cap C(M)$, there is some $e\in PC(M)\cap C(M)\cap
G(M)\setminus PT(M)$ and a curve $c:[0,1]\rightarrow E$ that is smooth w.r.t.
every $C^{k}$-compact open topology with $c(1) = e$ and with $c([0,1))\subset
F$.}
\end{enumerate}
\end{theorem}

\begin{proof}
Obviously $PT(M)\subset PC(M)$, so for the first assertion it is enough to see
that $PC(M)$ is closed. Take any $g\in\mathrm{Lor}(M)\setminus PC(M)$. For
every causal vector $v\in T_{p}M$ there exists a closed loop $c_{v}$ at $p$
such that the $g$-parallel transport along $c_{v}$ does not fix $v$, that is
$v\not =P_{c_{v}}^{g}(v)$. It is easy to see that still $v\not =P_{c_{v}}%
^{h}(v)$ for $h$ in an open neighborhood of $g$, where now $v$ may or may not
be an $h$-causal vector. Associated to $v$ we can take a tuple $(W_{v},V_{v})$
consisting of open neighborhoods of $g$ and $v$ respectively, small enough
such that $u\not =P_{c_{v}}^{h}(u)$ for every $h\in W_{v}$ and $u\in V_{v}$.
The set $L_{g}$ of $g$-causal vectors in $T_{p}M$ itself is not compact,
however, for every auxiliary scalar product in $T_{p}M$ and associated norm
$\left\vert \cdot\right\vert $, we can consider its unit sphere $S_{p}%
^{k}M:=\{v\in T_{p}M\ /\ \left\vert v\right\vert =1\}$, so $L_{g}\cap
S_{p}^{k}M$ is compact, and therefore covered by a finite number of open sets
$V_{v_{1}},...,V_{v_{k}}$. Take an open set $W\subset\cap_{i=1}^{k}W_{v_{i}}$
such that $g\in W$ and $V_{v_{1}},...,V_{v_{k}}$ still cover $L_{h}\cap
S_{p}^{k}M$ for every $h\in W $. If $h\in W$ and $v\in L_{h}$, there exists
$i$ such that $\frac{v}{\left\vert v\right\vert }\in V_{v_{i}}$ so
$P_{c_{v_{i}}}^{h}(v)\not =v$ because $h\in W\subset W_{v_{i}}$. This shows
that $W\subset\mathrm{Lor}(M)\setminus PC(M)$.

Now for the second part, assume $(M,g):=(\mathbb{R}_{t}\times S,\alpha\otimes
dt+dt\otimes\alpha+\overline{g})$ for a complete metric $\overline{g}$ on $S$
and a $\overline{g}$-bounded one-form $\alpha$ on $S$. Furthermore assume that
there is a point $x\in S$ with sectional curvature $k_{x}^{S}>0$. This can be
done with an arbitrarily small perturbation of a given metric in the $C^{k}%
$-compact open topology.

Define%
\[
c(r) :=-(1-t)dt^{2}+r(dt\otimes\alpha+\alpha\otimes dt)+\overline{g}%
\]
for $r\in\lbrack0,1]$ which is a continuous curve in $Lor(M)$, smooth w.r.t. every $C^k$-compact-open topology. One finds that
$t$ is a Cauchy time function for all $r$. In fact, it is easy to see that any
future vector $v$ has positive scalar product with $\mathrm{grad}_{c(r)}(t)$.

Let $k:\mathbb{R}\rightarrow M$ be a causal curve. Now, if $t\circ k$ is
bounded, it has a limit $t_{0}$ due to its monotonicity.

Now we parametrize $k$ according to $t$, that is, $k(t)=(t,\overline{k}),$ on
a bounded interval $[0,b)$. The Cauchy-Schwarz inequality implies that
$|\alpha(\overline{k})^{\prime}|\leq\left\vert \alpha\right\vert |\overline
{k}^{\prime}|$, the norm always being the one defined by $\overline{g}$.

Then, for $r=1$, using the causal character of $k$, we get $|\overline
{k}^{\prime}|\leq2\left\vert \alpha\right\vert $.

In case of $r<1$ we can solve the corresponding quadratic inequality for
$|\overline{k}^{\prime}|$ and get as a condition necessary for $c$ causal%

\[
|\overline{k}^{\prime}|\leq r\left\vert \alpha\right\vert +\sqrt
{r^{2}\left\vert \alpha\right\vert ^{2}+(1-r)}.
\]

Thus, by completeness of $\overline{g}$, also the $S$-coordinate along $k$ has
a limit at $b$, thus $t$ is Cauchy$,$ so $g_{r}\in G(M)$. Moreover, as
$\mathrm{grad}_{c(r)}(t)$ is $c(r)$-parallel, in particular it is $c(r)%
$-Killing, we have $c(r)(\mathrm{grad}_{c(r)}(t),k^{\prime})$ is constant
along any geodesic $k$. Thus $c(r)$ is a causally geodesically complete metric. This (and
the fact that $\mathrm{grad}_{c(r)}(t)$ is timelike for every $r\in
\lbrack0,1)$ and lightlike for $r=1)$, shows that $c([0,1))\subset F$ and
$c(1) \in E$.

Suppose now that $c(1) \in PT (M)$, that is, there exists a timelike $c(1)%
$-parallel vector field $Z\in\mathfrak{X}(M)$, in particular it is linearly
independent to $\frac{\partial}{\partial t}$ at any point. So there are
$c(1)$-degenerate planes $\pi$ in $T_{q}M$ for any point $q\in M$ such that
its null sectional curvature is zero, but this is not possible at points
$p=(t,x)\in M$ for any $t\in\mathbb{R}$ because by hypothesis ${\rm sec}_{x}^{S}>0$,
see \cite[Theorem 6.3 and Lemma 5.2]{GutOle09}. Contradiction.
\end{proof}

Now let us consider more closely the fine topologies. We want to argue in the
following that they are \emph{not} appropriate to consider spaces of metrics
of precompact holonomy. For a finite-dimensional bundle $\pi: E \rightarrow M
$, let $\Gamma^{k}( \pi)$ denote the space of sections of $\pi$ of regularity
$C^{k}$. The $C^{0}$-fine topology on $\Gamma^{0}(\pi)$ has as a neighborhood
basis of a section $f$ the family of sets $W_{U}:=\{\gamma\in\Gamma^{0}%
(\pi)\ /\ \gamma (M)\subset U\}$ where $U$ is an open neighborhood of $f\subset E$.
If $\pi$ is a vector bundle such that the fibers are locally convex metric
vector spaces with an arbitrary translational-invariant metric\footnote{Keep
in mind that here we use the word 'metric' not in the sense of bilinear form
but in the sense of distance on a metric space.} then we can describe the
topology in a different manner: Let $P$ be the space of smooth positive
functions on $M$, then, for $p\in P$, which could be called a \emph{profile
function}, we set%

\[
U_{p}:=\{f\in\Gamma^{0}(\pi)\ /\ d(f(x),0_{x})<p(x)\}
\]
where $0_{x}$ is the zero in $\pi^{-1}(x)$. Then $\{f+U_{p}\}$ is a
neighborhood basis for $f$ as well. The equivalence of these two descriptions
is easy to see, the arbitrariness of the auxiliary metric is compensated by
the flexible choice of the profile function. The $C^{k}$-fine topology is
defined by applying the same to the map $d^{k}\gamma$ as a section of the
bundle $S^{k}E\rightarrow S^{k}M$ where, for a manifold $N$, $S^{k}N$ is the
bundle of unit vectors in $T^{k}N$ for an arbitrary auxiliary Riemannian
metrics. For more details cf. \cite{BeemEhrlichEasley1996} and the references
therein. The choice of the fiber metric used is irrelevant due to local
compactness of the fiber.

The following theorem should be well-known to the experts, however we could
not find any reference in the literature and thus include a proof here:

\begin{theorem}
\label{CasiNoPasaNada} Let $\pi:E\rightarrow M$ be a metric vector bundle with
locally convex fibers over a finite-dimensional manifold. Let $a,b$ be two
$k$-times continuously differentiable sections of $\pi$. Then $a$ and $b$ are
in the same path connected component of $\Gamma^{k}(\pi)$ if and only if
$\mathrm{supp}(a-b)$ is compact.
\end{theorem}

\begin{proof}
As everything is translationally invariant, w.l.o.g. we can assume $b=0$, the
zero section. Assume the opposite of the statement of the theorem, that is,
there is a noncompactly supported section $a$ in the same path connected
component as $0$. By assumption, there is a $C^{0}$ curve $c:[0,1]\rightarrow
\Gamma^{k}(\pi)$ from $0$ to $a$. Choose $p_{n}\in\mathrm{supp}(a)$,
$p_{n}\rightarrow\infty$ (a sequence leaving every compact set) and define
$d_{n}:=d(a(p_{n}),0)>0$. Let $(C_{n})_{n\in\mathbb{N}}$ be a compact
exhaustion with $p_{n}\in C_{n+1}\setminus C_{n}$. And consider an open
neighborhood $W_{U}$ of $0$ as above with $U\cap\pi^{-1}(M\setminus
C_{n})\subset B_{d_{n}/n}$, for all $n$. As $c([0,1])$ is compact, it has a
finite covering by sets of the form $U_{i}:=c(t_{i})+U$, say $U_{1},...U_{m}$.
Then iterative application of the triangle inequality implies that
$d(a(p_{i}),0)\leq m\cdot d_{i}/i<d_{i}$ for $i>m$, contradiction.
\end{proof}

\vspace{0.5cm}

Now, the first corollary of the previous theorem is that $G(M)$ alone has
uncountably many path connected components each of which is intersected
nontrivially by $PT(M)$. This holds even if we mod out the action of the
diffeomorphism group on the space of metrics as it leaves the topology of the
Cauchy hypersurfaces unchanged.

\begin{corollary}
\label{Cauchy} Within each path connected component of $G(M)$ in
$\mathrm{Lor}(M)$ equipped with the $C^{0}$-fine topology, the topology of the
Cauchy surface does not vary. Consequently, for $M$ diffeomorphic to
$\mathbb{R}^{n}$ with $n\geq4$, the set $G(M)/\mathrm{Diff}(M)$ has
uncountably many path connected components, each of which contains elements of
$PT(M)$.
\end{corollary}

\begin{proof}
For the first assertion, single out two metrics $g_{1},g_{2}\in G(M)$ in the
same path connected component, then apply the previous theorem to
$\mathrm{Lor}(M)$ (equipped with any auxiliary Riemannian metric on the
fibers) to obtain that $g_{1}=g_{2}$ outside of a compact set $K$. Now, any
Cauchy surface which does not intersect $K$ is a Cauchy surface for either
metric. A recent result of Chernov-Nemirovski (\cite{CheNem13}, Remark 2.3)
states that for an open contractible manifold $C$ of dimension $n-1$, the
product $\mathbb{R}\times C$ is diffeomorphic to $\mathbb{R}^{n}$. Now equip
$C$ with a complete metric $g$ and consider the standard static manifold over
$(C,g)$. It is obviously diffeomorphic to $\mathbb{R}^{n}$. The Cauchy
surfaces, however, are diffeomorphic to $C$. As we know (see \cite{Wright92}
and the references therein) that for $n-1\geq3$, there are uncountably many
pairwise non-diffeomorphic contractible open manifolds (the Whitehead manifold
being an example for $n-1=3$), the statement follows.
\end{proof}

At this point, the reader probably is tempted to allow for an additional
compact factor $N$ to $\mathbb{R}^{n}$ and then to repeat the proof above.
However, this is not possible as the proof would yield $C_{1}\times N\cong
C_{2}\times N$ which could be true even for $C_{1}$ not homotopy equivalent to
$C_{2}$, for an example with $N=\mathbb{S}^{1}$ see \cite{Charlap65}. However,
we think that it should be possible to use the argument above for any
noncompact Cauchy surface, by replacing Whitehead's manifold suitably.

\begin{remark}
\label{R 5}It is well-known (see e.g. Corollary 7.32 and 7.37 in
\cite{BeemEhrlichEasley1996}) that causal completeness and causal
incompleteness are $C^{1}$-fine-stable properties, i.e., given a globally
hyperbolic causally complete resp. causally incomplete metric $g$, there is a
$C^{1}$-fine open neighborhood $U$ of $g$ such that all metrics $h\in U$ are
causally complete resp. causally incomplete. Using connectedness arguments we
get easily that each connected component of a globally hyperbolic metric
$g_{0}$ in the $C^{1}$-fine topology either consists entirely of causally
complete or consists entirely of causally incomplete metrics.
\end{remark}

The following corollary states that the $C^{0}$-fine topology is already too
fine for our purposes, as it isolates geometrically different metrics from
each other. Namely, if we focus on one of the uncountably many path connected
components, the result is only one $\mathrm{Diff}(M)$-orbit.

\begin{corollary}
\label{Corollary1}

\begin{enumerate}
\item {If $g_{0}\in PT(M)\cap G(M)$ is timelike complete, then any timelike
complete metric in the path connected component of $g_{0}$ in the $C^{0}$-fine
topology is isometric to $g_{0}$.}

\item {If $g_{0}\in PT(M)\cap G(M)$ is timelike complete, then any metric in
the path connected component of $g_{0}$ in the $C^{1}$-fine topology is
isometric to $g_{0}$.}
\end{enumerate}
\end{corollary}

\begin{proof}
Let $g_{1}$ be another element of $PT(M)\cap G(M)$ path-connected to $g_{0}$.
Both $g_{0}$ and $g_{1}$ admit global decompositions $I_{0}:(M,g_{0}%
)\rightarrow\mathbb{R}\times(S,h_{0})$ and $I_{1}:(M,g_{1})\rightarrow
\mathbb{R}\times(S,h_{1})$ (taking into account that the Cauchy surfaces of
$g_{0}$ are diffeomorphic to those of $g_{1}$ following Corollary
\ref{Cauchy}) with corresponding parallel vector fields $P^{0}=\mathrm{grad}%
^{g_{0}}(t_{0})$ resp. $P^{1}=\mathrm{grad}^{g_{1}}(t_{1})$ for associated
temporal functions $t_{0}$ resp. $t_{1}$. As the metrics coincide outside of a
compact set $K$ following Theorem \ref{CasiNoPasaNada}, the vector field
$P^{0}$ is $g_{0}$-parallel \textit{and} $g_{1}$-parallel on $(M\backslash
K,g_{1})$. Let us define $b:=\mathrm{sup}\{t_{0}(x)~/~x\in K\}$, then $P^{0}$
is in particular $g_{1}$-parallel and $g_{1}$-timelike on $t_{0}%
^{-1}((b,\infty))$.

Choose $a>b$. We want to construct an isometry between $(M,g_{1})$ and
$(\mathbb{R}\times S,-dt^{2}+h)$ where $h$ is the metric on $S:=t_{0}^{-1}(a)$
induced by the metric $g_{1}$ (or equivalently, by the metric $g_{0}$, as the
two are equal on $t_{0}^{-1}((b,\infty))$. Now we first show that $t_{0}%
^{-1}(a)$ is a Cauchy hypersurface of $(M,g_{1})$. So let a $C^{0}%
$-inextendible $g_{1}$-future curve $c$ be given. By the usual non-trapping
arguments, $c^{-1}(K)$ is compact. Let $s$ denote its maximum, then
$c|_{(s,\infty)}$ is a $g_{1}$-future causal curve in $M\setminus K$ that is
also a $g_{0}$-future causal curve as $g_{1}|_{M\setminus K}=g_{0}%
|_{M\setminus K}$. As moreover $\lim_{r\rightarrow s}c(r)\leq b<a$, we
conclude that the image of $c|_{(s,\infty)}$ intersects $t_{0}^{-1}(a)$. Thus,
indeed, $t_{0}^{-1}(a)=:S$ is a Cauchy surface for $(M,g_{1})$.

We define a vector field $Q^{0}$ by parallel transport from $S$ along $P^{1}$
with initial value $P^{0}$ on $S$, that is, $\nabla_{P^{1}}^{1}Q^{0}=0$ and
$Q^{0}|_{S}=P^{0}$ (where $\nabla^{1}:=\nabla^{g_{1}}$). This indeed defines a
vector field on $M$, as $S$ is a $g_{1}$-Cauchy surface and $P^{1}$ is a
complete future timelike vector field, thus its integral curves are $C^{0}%
$-inextendible future causal curves. Now we want to show that $\nabla^{1}%
Q^{0}=0$. So let $e_{i}$ be the $P^{1}$-parallel extension of a local
orthonormal basis of $TS$, then we have $\nabla_{P^{1}}^{1}e_{i}=[e_{i}%
,P^{1}]=\nabla_{e_{i}}^{1}P^{1}=0$. Moreover, the mixed curvature terms
vanish: $R^{1}(P^{1},W)V=0$ for any vectors $V,W$. Consequently, we get%

\[
\nabla_{P^{1}}^{1}\nabla_{e_{i}}^{1}Q^{0}=R^{1}(P^{1},e_{i})Q^{0}%
+\nabla_{e_{i}}^{1}\nabla_{P^{1}}^{1}Q^{0}+\nabla_{\lbrack P^{1},e_{i}]}%
^{1}Q^{0}=0,
\]
and the initial condition $Q^{0}|_{S}=P^{0}$ implies $\nabla_{e_{i}}^{1}%
Q^{0}|_{S}=0$, so the claim $\nabla^{1}Q^{0}=0$ follows. Now, as $Q^{0}$ is
$g^{1}$-parallel and $g^{1}$ is timelike complete, the flow of $Q^{0}$ is
complete, any integral function of $Q^{0}$ and the flow of $Q^{0}$ define an
isometry between $(M,g_{1})$ and $(\mathbb{R}\times S,-dt^{2}+h)$ (note that
$S$ is always a level set for any integral function of $Q^{0}$ as $Q^{0}$ is
orthogonal on $S$).

In fact, if $\Phi:\mathbb{R}\times S\rightarrow M$ is the flow of $Q^{0}$,
identifiying $\{0\}\times S$ with $S\subset M$, using that $Q^{0}$ is $g_{1}%
$-timelike, complete and $S$ a Cauchy hypersurface for $g_{1}$, it is clear
that $\Phi$ is a diffeomorphism$.$ Using that $Q^{0}$ is $g_{1}$-parallel it
is clear that $\Phi$ is an isometry. Moreover, it sends $\frac{\partial
}{\partial t}$ to $Q^{0}$.

As the same is true for $g_{0}$ via the function $t_{0}$, the metrics are isometric.

The second assertion follows from the first part and from the observation
above that the whole path connected component of $g_{0}$ consists of causally
complete metrics, see Remark \ref{R 5}.
\end{proof}

\bigskip

It remains as an interesting question to examine the topology of subsets of
globally hyperbolic metrics with holonomy of certain kinds for other manifold
structures on $\mathrm{Lor}(M)$, possibly not coming from a vector space
topology on $\mathrm{Bil}(M)$.


\begin{thebibliography}{99}                                                                                               %


\bibitem {AmoGut99}A. M. Amores and M. Guti\'{e}rrez: The b-completion of the
Friedmann space, \textit{J. Geom. Phys}. \textbf{29} (1999) 177-197.

\bibitem {Ake2016}Luis Alberto Ak\'{e} Hau, Miguel S\'{a}nchez: Compact
affine manifolds with precompact holonomy are geodesically complete.
\textit{J. Math. Anal. Appl.} \textbf{436} (2016) 1369--1371.

\bibitem {hB2012}Helga Baum: Holonomy groups of Lorentzian manifolds: a status
report. Global differential geometry, 163-200, Springer Proc. Math. \textbf{17}
(2012), Springer, Heidelberg.

\bibitem {BeemEhrlichEasley1996}J. K. Beem, P. E. Ehrlich and K. L. Easley:
Global Lorentzian geometry, Second Edition. Marcel Dekker, Inc. 1996.

\bibitem {BerIke1993}L. B\'{e}rard-Bergery and A. Ikemakhen: On the holonomy
of Lorentzian manifolds, in 'Differential Geometry: Geometry in Mathematical
Physics and Related Topics'. (Los Angeles, Ca, 1990), \textit{Proc. Sympos.
Pure Math}. \textbf{54}, 27-40, Amer. Math. Soc., Providence, RI, 1993.

\bibitem {Besse1987}A. L. Besse: Einstein Manifolds. Springer-Verlag, New York 1987.

\bibitem {BorLic52}A. Borel and A. Lichnerowicz: Groupes d'holonomie de
vari\'{e}t\'{e}s riemanniennes, \textit{C. R. Math. Acad. Sci. Paris},
\textbf{234}, (1952) 1835-1837.

\bibitem {Charlap65}Leonard S. Charlap: Compact Flat Riemannian Manifolds: I,
\textit{Ann. Math}. \textbf{81}, (1965) 15-30.

\bibitem {Chen2014}Z. Chen: The Uniqueness in the de Rham-Wu Decomposition,
\textit{J. Geom. Anal.} \textbf{25}, no 4, 2687 --- 2697 (2015).

\bibitem {CheNem13}Vladimir Chernov, Stefan Nemirovski: Cosmic censorship of
smooth structures, \textit{Comm. Math. Phys}. \textbf{320} (2013) 469-473.

\bibitem {aGtL}Anton Galaev, Thomas Leistner: Holonomy groups of Lorentzian
manifolds: classification, examples, and applications. Recent developments in
pseudo-Riemannian geometry, 53-96, ESI Lect. Math. Phys., Eur. Math. Soc.,
Z\"urich (2008)

\bibitem {aG2015}Anton Galaev: Holonomy groups of Lorentzian manifolds.
(Russian) Uspekhi Mat. Nauk \textbf{70} (2015), no. 2(422), 55--108; translation in
Russian Math. Surveys \textbf{70} (2015), no. 2, 249-298

\bibitem {GHP91}J. Gruszczak, M. Heller, Z. Pogoda: Cauchy boundary and
b-incompleteness of Space-Time.\ \textit{International Journal of Theoretical
Physics}. \textbf{30} (1991) 555-565.

\bibitem {GutOle09}M. Guti\'{e}rrez and B. Olea: Global decomposition of a
manifold as a Generalized Robertson-Walker space, \textit{Differential Geom.
Appl}. \textbf{27}, (2009) 146-156.

\bibitem {GutOle12}M. Guti\'{e}rrez and B. Olea: Semi-Riemannian manifolds
with a doubly warped structure, \textit{Rev. Mat. Iberoam}. \textbf{28} (2012) 1--24.

\bibitem {Gut09}M. Guti\'{e}rrez: Equivalence of Cauchy singular boundary and
b-boundary in $O(3)$-reducible space times, \textit{J. Geom. Phys}.
\textbf{59} (2009) 1196-1198.

\bibitem {Har82}S. G. Harris: A triangle comparison theorem for Lorentz
manifolds, \textit{Indiana Univ. Math. J}. \textbf{31}, (1982) 289--308.

\bibitem {Harris1985}S. G. Harris: A characterization of Robertson-Walker
spaces by null sectional curvature, \textit{Gen. Relativity Gravitation},
\textbf{17}, (1985) 493-498.

\bibitem {KobayNomiz1963vI}S. Kobayashi and K. Nomizu: Foundations of
differential geometry, Volume I. Interscience Publishers, (1963)

\bibitem {LeistnerSchliebner13}T. Leistner and D. Schliebner: Completeness of
compact Lorentzian manifolds with special holonomy, Mathematische Annalen (to
appear, 2016). http://dx.doi.org/10.1007/s00208-015-1270-4.
\textit{arXiv:1306.0120v2}.

\bibitem {Marsden1972}J. E. Marsden: On completeness of homogeneous
pseudo-Riemannian manifolds, \textit{Indiana Univ. J}. \textbf{22} (1972-1973) 1065-1066.

\bibitem {jN}Johan Noldus: A new topology on the space of Lorentzian metrics
on a fixed manifold, Class. Quant. Grav. \textbf{19}, no 23, 6075 --- 6107 (2002).
\textit{arxiv.org}: 1104.1811.

\bibitem {RS}Alfonso Romero, Miguel S\'{a}nchez: Completeness of compact
Lorentz manifolds admitting a timelike conformal Killing vector field,
\textit{Proc AMS}, \textbf{123} (1995) 2831-2833.

\bibitem {Schm71}B. G. Schmidt: A new definition of singular points in general
relativity, \textit{Gen. Relativity Gravitation} \textbf{1}, (1971) 269-280.

\bibitem {Wil99}B. Wilking: On compact Riemannian manifolds with noncompact
holonomy groups, \textit{J. Diff. Geom}. \textbf{52}, (1999) 223-257.

\bibitem {Wright92}David G. Wright: Contractible open manifolds which are not
covering spaces, \textit{Topology} vol. \textbf{31}, (1992) 281-291.

\bibitem {Wu64}H. Wu: On the De Rham decomposition theorem, \textit{Illinois
J. Math.} \textbf{8} (1964) 291-311.

\bibitem {Wu67}H. Wu: Holonomy groups of indefinite metrics, \textit{Pacific
J. Math}. \textbf{20}, (1967) 351-392.
\end{thebibliography}
\end{document}